\documentclass[11pt,draft]{article}
\usepackage{amsmath,amscd}
\usepackage{amssymb,latexsym,amsthm}
\usepackage{color}
\usepackage[spanish,english]{babel}
\usepackage{amssymb}
\usepackage{color}
\usepackage{amsmath,amsthm,amscd}
\usepackage[latin1]{inputenc}
\usepackage{lscape}
\usepackage{amsfonts}

\numberwithin{equation}{section}
\newtheorem{theorem}{Theorem}[section]

\newtheorem{proposition}[theorem]{Proposition}

\theoremstyle{definition}

\newtheorem{definition}[theorem]{Definition}
\newtheorem{examples}[theorem]{Examples}

\newtheorem{remark}[theorem]{Remark}

\newcommand{\cA}{\mbox{${\cal A}$}}

\newcommand{\cU}{\mbox{${\cal U}$}}

\title{\textbf{Skew Poincar\'e-Birkhoff-Witt extensions\\ over weak $\Sigma$-rigid rings}}
\author{Armando Reyes\footnote{Departamento de  Matem\'aticas. e-mail: mareyesv@unal.edu.co} \\ Universidad Nacional de Colombia, Sede Bogot\'a, Colombia\\ H\'ector Su\'arez\footnote{Escuela de Matem\'aticas y Estad\'istica. e-mail: hector.suarez@uptc.edu.co} \\ Universidad Pedag\'ogica y Tecnol\'ogica de Colombia, Sede Tunja, Colombia}
\date{}
\begin{document}
\maketitle
\begin{abstract}
\noindent In this paper we introduce the notion of weak $\Sigma$-rigid ring which extends $\alpha$-rigid rings and $\Sigma$-rigid rings defined for Ore extensions and skew PBW extensions, respectively. We also present the notion of weak $\Sigma$-skew Armendariz ring which extends that of $\alpha$-skew Armendariz ring. In this way we generalize several results in the literature, from Ore extensions of injective type to the more general setting of skew PBW extensions.

\bigskip

\noindent \textit{Key words and phrases.} Rigid rings, Armendariz rings, skew PBW extensions.

\bigskip

\noindent 2010 \textit{Mathematics Subject Classification:} 16S36, 16T20, 16S30.
\bigskip

\end{abstract}
\section{Introduction}\label{introduccion}
In commutative algebra, a ring $B$ is called {\em Armendariz} (the term was introduced by Rege and Chhawchharia in \cite{RegeChhawchharia1997}), if whenever polynomials $f(x)=a_0+a_1x+\dotsb + a_nx^n$, $g(x)=b_0+b_1x+\dotsb + b_mx^m\in B[x]$ such that  $f(x)g(x)=0$, then $a_ib_j=0$, for every $i,j$. The interest of this notion lies in its natural and its useful role in understanding the relation between the annihilators of the ring $B$ and the annihilators of the polynomial ring $B[x]$. As a matter of fact, in  \cite{Armendariz1974}, Lemma 1, Armendariz showed that a  reduced ring (a ring has no nonzero nilpotent elements) always satisfies this condition. It is well known that reduced rings are abelian (i.e., every idempotent is central). Now, following \cite{LiuZhao2006}, a ring $B$ is called {\em weak Armendariz}, if whenever two polynomials $p(x)=\sum_{i=0}^{s} a_ix^{i}$ and $q(x)=\sum_{j=0}^{t} b_jx^{j}$ of the polynomial ring $B[x]$ satisfy $pq=0$, then $a_ib_j$ is a nilpotent element of $B$, for each $i, j$. \\

In the context of Ore extensions introduced by Ore in \cite{Ore1933}, for $\alpha$ an endomorphism of a ring $B$, Hong et al. \cite{HongKimKwak2003} called $B$ an $\alpha$-{\em skew Armendariz ring}, if for two elements $p=\sum_{i=0}^{s} a_ix^{i},\ q=\sum_{j=0}^{t} b_jx^{j}$ of the Ore extension of endomorphism type $B[x;\alpha]$, $pq = 0 \Rightarrow a_i\sigma^{i}(b_j)=0$, for every $i, j$. As a generalization of the $\alpha$-skew Armendariz rings, Ouyang \cite{Ouyang2008} defined the {\em weak} $\alpha$-{\em skew Armendariz rings} in the following way: a ring $B$ is said to be {\em weak}  $\alpha$-{\em skew Armendariz rings}, if whenever two polynomials $p=\sum_{i=0}^{s} a_ix^{i},\ q=\sum_{j=0}^{t} b_jx^{j}$ of $B[x;\alpha]$, $pq=0 \Rightarrow  a_i\sigma^{i}(b_j)$ is a nilpotent element of $B$, for all $i, j$. It is clear that weak $\alpha$-skew Armendariz rings are more general than weak Armendariz rings and $\alpha$-skew Armendariz rings.\\

On the other hand, following Krempa \cite{Krempa1996}, an endomorphism $\alpha$ of a ring $B$ is called {\em rigid}, if $a\alpha(a)=0\Rightarrow a=0$, for $a\in B$. $B$ is called $\alpha$-rigid, if there exists a rigid endomorphism $\alpha$ of $B$. It is known that any rigid endomorphism of a ring is injective and $\alpha$-rigid rings are reduced (see Hong et al \cite{HongKimKwak2000}). Several properties of $\alpha$-rigid rings have been established in the literature (c.f. \cite{Krempa1996}, \cite{HongKimKwak2000}, and see \cite{Reyes2015} for detailed references). With this definition in mind, Ouyang \cite{Ouyang2008} defined {\em weak} $\alpha$-{\em rigid rings} which are a generalization of $\alpha$-rigid rings. More precisely, if $\alpha$ is en endomorphism of a ring $B$, $B$ is said to be {\em weak} $\alpha$-{\em rigid}, if $a\alpha(a)\in {\rm nil}(R) \Leftrightarrow a\in {\rm nil}(B)$, where ${\rm nil}(B)$ is the set of nilpotent elements of $B$. Ouyang \cite{Ouyang2008}, Proposition 2.2, showed that $B$ is $\alpha$-rigid if and only if $B$ is weak $\alpha$-rigid and reduced. In this way, weak $\alpha$-rigid rings are a generalization of $\alpha$-rigid rings deleting the condition to be reduced.\\

With the aim of extending the above two notions introduced by Ouyang in \cite{Ouyang2008} to a more general setting than Ore extensions, in this paper we focus on the kind of noncommutative rings known in the literature as skew Poincar\'e-Birkhoff-Witt extensions  (briefly, skew PBW extensions). These objects were introduced by Gallego and Lezama in \cite{LezamaGallego2011} and contain strictly Ore extensions of injective type (i.e., when $\sigma$ is an injective endomorphism of $R$; see Example \ref{mentioned} for different noncommutative rings which are skew PBW extensions but they can not be expressed as Ore extensions). As a matter of fact,  skew PBW extensions generalize several families of noncommutative rings defined in the literature, and include as particular rings different examples of remarkable algebras appearing in mathematical physics, representation theory, Hopf algebras, quantum groups, Lie algebras, and others. Briefly, next we mention some of these families of algebras (see \cite{Reyes2013PhD} and \cite{LezamaReyes2014} for a detailed reference of every one of these families): (i) universal enveloping algebras of finite dimensional Lie algebras; (ii) PBW extensions introduced by Bell and Goodearl; (iii) almost normalizing extensions defined by McConnell and Robson; (iv) sol\-va\-ble polynomial rings introduced by Kandri-Rody and Weispfenning; (v) diffusion algebras studied by Isaev, Pyatov, and Rittenberg; (vi) 3-dimensional skew polynomial algebras introduced by Bell and Smith; (vii) the regular graded algebras studied by Kirkman, Kuzmanovich, and Zhang; (viii) different algebras with PBW bases of polynomial type. The greatest difference of skew PBW extensions with respect all these algebras is that the coefficients do not necessarily commute with the variables, and these coefficients are not necessarily elements of fields (see Definition \ref {gpbwextension} below). Due to this fact, the skew PBW extensions contain well-known groups of algebras such as some types of $G$-algebras studied by Levandovskyy and some PBW algebras defined by Bueso et. al., (both $G$-algebras and PBW algebras take coefficients in fields and assume that coefficientes commute with variables), Auslander-Gorenstein rings, some Calabi-Yau and skew Calabi-Yau algebras, some Artin-Schelter regular algebras, some Koszul algebras, quantum polynomials, some quantum universal enveloping algebras (see \cite{LezamaReyes2014}, \cite{ReyesSuarez2017FEJM}, \cite{ReyesSuarezMomento2017}, \cite{SuarezLezamaReyes2017}, \cite{SuarezReyes2017JP} and \cite{SuarezReyes2017FJMS} for a considerable list of examples of all these algebras). As we see, skew PBW extensions include a lot of noncommutative rings, which means that a theory extending the two notions above for these extensions will establish general results for a lot of noncommutative rings much more general than Ore extensions, and hence we contribute to the study of properties of noncommutative algebras. Precisely, this kind of thinking has been presented a few years ago in several works (c.f.  \cite{ReyesSuarezClifford2017}, \cite{ReyesSuarezskewCY2017}, \cite{ReyesSuarezUMA2018} and \cite{ReyesSuarezYesica2018}).\\

The paper is organized as follows: In Section \ref{definitionexamplesspbw} we establish some useful results about skew PBW extensions for the rest of the paper. Section \ref{weakSigmarigid} contains the first concept of the paper, the {\em weak} $\Sigma$-{\em rigid rings} (Definition \ref{weaksigmarigidring}). These rings are a generalization of weak $\alpha$-rigid rings introduced by Ouyang \cite{Ouyang2008} and $\Sigma$-rigid rings defined by the first author in \cite{Reyes2015}. However, as we will see in Theorem \ref{2008Proposition2.2}, weak $\Sigma$-rigid rings and $\Sigma$-rigid rings coincide when the ring is assumed to be reduced. Different results of \ref{weakSigmarigid} are presented in this section. In Section \ref{weakSigmaskewArmendariz} we present the second concept of this paper, the {\em weak} $\Sigma$-{\em skew Armendariz rings} (Definition \ref{weakSigmaskewArmendariz}) which are a generalization of weak $\alpha$-skew Armendariz rings defined by Ouyang \cite{Ouyang2008} and $\Sigma$-skew Armendariz rings introduced by the first author in \cite{ReyesSuarez2016UIS}. We prove that when $R$ is a NI ring (a ring $B$ is called NI ring, if the set ${\rm nil}(B)$ of nilpotent elements of $B$ forms an ideal of $B$), if $R$ is weak $\Sigma$-rigid ring, then $R$ is a weak $\Sigma$-skew Armendariz ring (Theorem \ref{2008Theorem3.3}). We present an example which illustrates the importance of the condition NI (if we do not assume this fact, then there exist examples of rings which are weak $\Sigma$-rigid but not weak $\Sigma$-skew Armendariz, see Remark \ref{2008Example3.4}). Finally, Section \ref{futurework} presents some ideas for a future work concerning the objects introduced in this article.\\

Throughout the paper, the word ring means a ring (not necessarily commutative) with unity. $\mathbb{C}$ will denote the field of complex numbers.
\section{Skew PBW extensions}\label{definitionexamplesspbw}
In this section we establish some useful results about skew PBW extensions for the rest of the paper.
\begin{definition}[\cite{LezamaGallego2011}, Definition 1]\label{gpbwextension}
Let $R$ and $A$ be rings. We say that $A$ is a {\em skew PBW extension  over} $R$, which is denoted by $A:=\sigma(R)\langle
x_1,\dots,x_n\rangle$, if the following conditions hold:
\begin{enumerate}
\item[\rm (i)]$R\subseteq A$;
\item[\rm (ii)]there exist elements $x_1,\dots ,x_n\in A$ such that $A$ is a left free $R$-module, with basis ${\rm Mon}(A):= \{x^{\alpha}=x_1^{\alpha_1}\cdots
x_n^{\alpha_n}\mid \alpha=(\alpha_1,\dots ,\alpha_n)\in
\mathbb{N}^n\}$,  and $x_1^{0}\dotsb x_n^{0}:=1\in {\rm Mon}(A)$.

\item[\rm (iii)]For each $1\leq i\leq n$ and any $r\in R\ \backslash\ \{0\}$, there exists an element $c_{i,r}\in R\ \backslash\ \{0\}$ such that $x_ir-c_{i,r}x_i\in R$.
\item[\rm (iv)]For any elements $1\leq i,j\leq n$, there exists $c_{i,j}\in R\ \backslash\ \{0\}$ such that $x_jx_i-c_{i,j}x_ix_j\in R+Rx_1+\cdots +Rx_n$.
\end{enumerate}
\end{definition}
\begin{proposition}[\cite{LezamaGallego2011}, Proposition
3]\label{sigmadefinition}
Let $A$ be a skew PBW extension over $R$. For each $1\leq i\leq
n$, there exist an injective endomorphism $\sigma_i:R\rightarrow
R$ and an $\sigma_i$-derivation $\delta_i:R\rightarrow R$ such that $x_ir=\sigma_i(r)x_i+\delta_i(r)$, for  each $r\in R$. From now on, we write  $\Sigma:=\{\sigma_1,\dotsc, \sigma_n\}$, and $\Delta:=\{\delta_1,\dotsc, \delta_n\}$.
\end{proposition}
\begin{definition}[\cite{LezamaGallego2011}, Definition 4]\label{sigmapbwderivationtype}
Let $A$ be a skew PBW extension over $R$.
\begin{enumerate}
\item[\rm (i)] $A$ is called \textit{quasi-commutative} if the conditions {\rm(}iii{\rm)} and {\rm(}iv{\rm)} in Definition \ref{gpbwextension} are replaced by the following: (iii') for each $1\leq i\leq n$ and all $r\in R\ \backslash\ \{0\}$, there exists $c_{i,r}\in R\ \backslash\ \{0\}$ such that $x_ir=c_{i,r}x_i$; (iv') for any $1\leq i,j\leq n$, there exists $c_{i,j}\in R\ \backslash\ \{0\}$ such that $x_jx_i=c_{i,j}x_ix_j$.
\item[\rm (ii)] $A$ is called \textit{bijective}, if $\sigma_i$ is bijective for each $1\leq i\leq n$, and $c_{i,j}$ is invertible, for any $1\leq i<j\leq n$.
\item[\rm (iii)] $A$ is called of {\em endomorphism type}, if $\delta_i=0$, for every $i$.  In addition, if every $\sigma_i$ is bijective, $A$ is a skew PBW extension of {\em automorphism type}.
\end{enumerate}
\end{definition}
\begin{examples}\label{mentioned}
If $R[x_1;\sigma_1,\delta_1]\dotsb [x_n;\sigma_n,\delta_n]$ is an iterated Ore extension where
\begin{itemize}
\item $\sigma_i$ is injective, for $1\le i\le n$;
\item $\sigma_i(r)$, $\delta_i(r)\in R$, for every $r\in R$ and $1\le i\le n$;
\item $\sigma_j(x_i)=cx_i+d$, for $i < j$, and $c, d\in R$, where $c$ has a left inverse;
\item $\delta_j(x_i)\in R + Rx_1 + \dotsb + Rx_n$, for $i < j$,
\end{itemize}
then $R[x_1;\sigma_1,\delta_1]\dotsb [x_n;\sigma_n, \delta_n] \cong \sigma(R)\langle x_1,\dotsc, x_n\rangle$ (\cite{LezamaReyes2014}, p. 1212). Note that skew PBW extensions of endomorphism type are more general than iterated Ore extensions $R[x_1;\sigma_1]\dotsb [x_n;\sigma_n]$, and in general, skew PBW extensions are more general than Ore extensions of injective type. More precisely, next we show that there are noncommutative rings which are skew PBW extensions but they can not be expressed as iterated Ore extensions as we will next (see \cite{Reyes2013PhD},  \cite{LezamaReyes2014} or \cite{ReyesSuarezClifford2017} for the reference of every example). Examples of these extensions appearing in noncommutative algebraic geometry and theoretical physics can be found in  \cite{ReyesSuarez2017FEJM}, \cite{ReyesSuarezMomento2017}, \cite{ReyesSuarezskewCY2017}, \cite{ReyesSuarezYesica2018}, \cite{SuarezReyes2017JP} and \cite{SuarezReyes2017FJMS}.
 
\begin{enumerate}
\item[\rm (a)] Let $k$ be a commutative ring and $\mathfrak{g}$ a finite dimensional Lie algebra over $k$ with basis $\{x_1,\dots ,x_n\}$. The \textit{universal enveloping algebra} of $\mathfrak{g}$, denoted $\cU(\mathfrak{g})$, is a skew PBW extension over $k$, since $x_ir-rx_i=0$, $x_ix_j-x_jx_i=[x_i,x_j]\in \mathfrak{g}=k+kx_1+\cdots+kx_n$, $r\in k$, for $1\leq i,j\leq n$. In particular, the \textit{universal enveloping algebra} \textit{of a Kac-Moody Lie algebra} is a skew PBW extension over a polynomial ring.
\item [\rm (b)] The \textit{universal enveloping ring} $\cU(V,R ,\Bbbk)$, where  $R$ is a $\Bbbk$-algebra, and $V$ is a $\Bbbk$-vector space which is also a Lie ring containing $R$ and $\Bbbk$ as Lie ideals with suitable relations. The enveloping ring $\cU(V,R,\Bbbk)$ is a finite skew PBW extension over $R$ if ${\rm dim}_\Bbbk\ (V/R)$ is finite.
\item [\rm (c)] Let $k$, $\mathfrak{g}$, $\{x_1,\dots ,x_n\}$ and $\cU(\mathfrak{g})$ be as in the previous example; let $R$ be a $k$-algebra containing $k$. The \textit{tensor product} $A:=R\ \otimes_k\ \cU(\mathfrak{g})$ is a skew PBW extension over $R$, and it is a particular case of \textit{crossed product} $R*\cU(\mathfrak{g})$ of $R$ by $\cU(\mathfrak{g})$, which is a skew PBW extension over $R$.
\item [\rm (d)] The \textit{twisted or smash product differential operator ring} $R\ \# _{\sigma}\  \cU(\mathfrak{g})$, where $\mathfrak{g}$ is a finite-dimensional Lie algebra acting on $R$ by derivations, and $\sigma$ is Lie 2-cocycle with values in $R$.
\item [\rm (e)] Diffusion algebras arise in physics as a possible way to understand a large class of $1$-dimen\-sional stochastic process. A \textit{diffusion algebra} $\cA$ with parameters $a_{ij}\in \mathbb{C}\ \backslash\ \{0\}$ for $1\le i, j\le n$, is an algebra over $\mathbb{C}$ generated by variables $x_1,\dotsc,x_n$ subject to relations $a_{ij}x_ix_j-b_{ij}x_jx_i=r_jx_i-r_ix_j$, whenever $i<j$, $b_{ij}, r_i\in \mathbb{C}$ for all $i<j$. $\cA$ admits a $PBW$-basis of standard monomials $x_1^{i_1}\dotsb x_n^{i_n}$, that is, $\cA$ is a diffusion algebra if these standard monomials are a $\mathbb{C}$-vector space basis for $\cA$. From Definition \ref{gpbwextension}, (iii) and (iv), it is clear that the family of skew PBW extensions are more general than diffusion algebras.  We will denote $q_{ij}:=\frac{b_{ij}}{a_{ij}}$. The parameter $q_{ij}$ can be a root of unity if and only if is equal to 1. It is therefore reasonable to assume that these parameters not to be a root of unity other than 1. If all coefficients $q_{ij}$ are nonzero, then the corresponding diffusion algebra have a PBW basis of standard monomials $x_1^{i_1}\dotsb x_n^{i_n}$, and hence these algebras are skew PBW extensions. More precisely, $\cA$ is a skew PBW extension over $\mathbb{C}$ with indetermnates $x_1,\dotsc,x_n$.
\end{enumerate}
\end{examples}
\begin{definition}[\cite{LezamaGallego2011}, Definition 6]\label{definitioncoefficients}
Let $A$ be a skew PBW extension over $R$. Then:
\begin{enumerate}
\item[\rm (i)]for $\alpha=(\alpha_1,\dots,\alpha_n)\in \mathbb{N}^n$,
$\sigma^{\alpha}:=\sigma_1^{\alpha_1}\cdots \sigma_n^{\alpha_n}$,
$|\alpha|:=\alpha_1+\cdots+\alpha_n$. If
$\beta=(\beta_1,\dots,\beta_n)\in \mathbb{N}^n$, then
$\alpha+\beta:=(\alpha_1+\beta_1,\dots,\alpha_n+\beta_n)$.
\item[\rm (ii)]For $X=x^{\alpha}\in {\rm Mon}(A)$,
$\exp(X):=\alpha$, $\deg(X):=|\alpha|$, and $X_0:=1$. The symbol $\succeq$ will denote a total order defined on ${\rm Mon}(A)$ (a total order on $\mathbb{N}^n$). For an
 element $x^{\alpha}\in {\rm Mon}(A)$, ${\rm exp}(x^{\alpha}):=\alpha\in \mathbb{N}^n$.  If
$x^{\alpha}\succeq x^{\beta}$ but $x^{\alpha}\neq x^{\beta}$, we
write $x^{\alpha}\succ x^{\beta}$. Every element $f\in A$ can be expressed uniquely as $f=a_0 + a_1X_1+\dotsb +a_mX_m$, with $a_i\in R$, and $X_m\succ \dotsb \succ X_1$ (eventually, we will use expressions as $f=a_0 + a_1Y_1+\dotsb +a_mY_m$, with $a_i\in R$, and $Y_m\succ \dotsb \succ Y_1$). With this notation, we define ${\rm
lm}(f):=X_m$, the \textit{leading monomial} of $f$; ${\rm
lc}(f):=a_m$, the \textit{leading coefficient} of $f$; ${\rm
lt}(f):=a_mX_m$, the \textit{leading term} of $f$; ${\rm exp}(f):={\rm exp}(X_m)$, the \textit{order} of $f$; and
 $E(f):=\{{\rm exp}(X_i)\mid 1\le i\le t\}$. Note that $\deg(f):={\rm max}\{\deg(X_i)\}_{i=1}^t$. Finally, if $f=0$, then
${\rm lm}(0):=0$, ${\rm lc}(0):=0$, ${\rm lt}(0):=0$. We also
consider $X\succ 0$ for any $X\in {\rm Mon}(A)$. For a detailed description of monomial orders in skew PBW  extensions, see \cite{LezamaGallego2011}, Section 3.
\end{enumerate}
\end{definition}
\begin{proposition}[\cite{Reyes2015}, Proposition 2.9]  \label{lindass}
If $\alpha=(\alpha_1,\dotsc, \alpha_n)\in \mathbb{N}^{n}$ and $r$ is an element of $R$, then  
\begin{align*}
x^{\alpha}r = &\ x_1^{\alpha_1}x_2^{\alpha_2}\dotsb x_{n-1}^{\alpha_{n-1}}x_n^{\alpha_n}r = x_1^{\alpha_1}\dotsb x_{n-1}^{\alpha_{n-1}}\biggl(\sum_{j=1}^{\alpha_n}x_n^{\alpha_{n}-j}\delta_n(\sigma_n^{j-1}(r))x_n^{j-1}\biggr)\\
+ &\ x_1^{\alpha_1}\dotsb x_{n-2}^{\alpha_{n-2}}\biggl(\sum_{j=1}^{\alpha_{n-1}}x_{n-1}^{\alpha_{n-1}-j}\delta_{n-1}(\sigma_{n-1}^{j-1}(\sigma_n^{\alpha_n}(r)))x_{n-1}^{j-1}\biggr)x_n^{\alpha_n}\\
+ &\ x_1^{\alpha_1}\dotsb x_{n-3}^{\alpha_{n-3}}\biggl(\sum_{j=1}^{\alpha_{n-2}} x_{n-2}^{\alpha_{n-2}-j}\delta_{n-2}(\sigma_{n-2}^{j-1}(\sigma_{n-1}^{\alpha_{n-1}}(\sigma_n^{\alpha_n}(r))))x_{n-2}^{j-1}\biggr)x_{n-1}^{\alpha_{n-1}}x_n^{\alpha_n}\\
+ &\ \dotsb + x_1^{\alpha_1}\biggl( \sum_{j=1}^{\alpha_2}x_2^{\alpha_2-j}\delta_2(\sigma_2^{j-1}(\sigma_3^{\alpha_3}(\sigma_4^{\alpha_4}(\dotsb (\sigma_n^{\alpha_n}(r))))))x_2^{j-1}\biggr)x_3^{\alpha_3}x_4^{\alpha_4}\dotsb x_{n-1}^{\alpha_{n-1}}x_n^{\alpha_n} \\
+ &\ \sigma_1^{\alpha_1}(\sigma_2^{\alpha_2}(\dotsb (\sigma_n^{\alpha_n}(r))))x_1^{\alpha_1}\dotsb x_n^{\alpha_n}, \ \ \ \ \ \ \ \ \ \ \sigma_j^{0}:={\rm id}_R\ \ {\rm for}\ \ 1\le j\le n.
\end{align*}
\end{proposition}
\begin{remark}[\cite{Reyes2015}, Remark 2.10)]\label{juradpr}
About Proposition \ref{lindass}, we have the following observation: If $X_i:=x_1^{\alpha_{i1}}\dotsb x_n^{\alpha_{in}}$ and $Y_j:=x_1^{\beta_{j1}}\dotsb x_n^{\beta_{jn}}$, then when we compute every summand of $a_iX_ib_jY_j$ we obtain pro\-ducts of the coefficient $a_i$ with several evaluations of $b_j$ in $\sigma$'s and $\delta$'s depending of the coordinates of $\alpha_i$. This assertion follows from the expression:
\begin{align*}
a_iX_ib_jY_j = &\ a_i\sigma^{\alpha_{i}}(b_j)x^{\alpha_i}x^{\beta_j} + a_ip_{\alpha_{i1}, \sigma_{i2}^{\alpha_{i2}}(\dotsb (\sigma_{in}^{\alpha_{in}}(b_j)))} x_2^{\alpha_{i2}}\dotsb x_n^{\alpha_{in}}x^{\beta_j} \\
+ &\ a_i x_1^{\alpha_{i1}}p_{\alpha_{i2}, \sigma_3^{\alpha_{i3}}(\dotsb (\sigma_{{in}}^{\alpha_{in}}(b_j)))} x_3^{\alpha_{i3}}\dotsb x_n^{\alpha_{in}}x^{\beta_j} \\
+ &\ a_i x_1^{\alpha_{i1}}x_2^{\alpha_{i2}}p_{\alpha_{i3}, \sigma_{i4}^{\alpha_{i4}} (\dotsb (\sigma_{in}^{\alpha_{in}}(b_j)))} x_4^{\alpha_{i4}}\dotsb x_n^{\alpha_{in}}x^{\beta_j}\\
+ &\ \dotsb + a_i x_1^{\alpha_{i1}}x_2^{\alpha_{i2}} \dotsb x_{i(n-2)}^{\alpha_{i(n-2)}}p_{\alpha_{i(n-1)}, \sigma_{in}^{\alpha_{in}}(b_j)}x_n^{\alpha_{in}}x^{\beta_j} \\
+ &\ a_i x_1^{\alpha_{i1}}\dotsb x_{i(n-1)}^{\alpha_{i(n-1)}}p_{\alpha_{in}, b_j}x^{\beta_j}.
\end{align*}
\end{remark}
\section{Weak $\Sigma$-rigid rings}\label
{weakSigmarigid}
For a ring $B$ with a ring endomorphism $\sigma:B\to B$, an $\sigma$-derivation $\delta:B\to B$, considering the Ore extension $B[x;\sigma,\delta]$, Krempa in \cite{Krempa1996} defined $\sigma$ as a  {\em rigid endomorphism} if $b\sigma(b)=0$ implies $b=0$, for $b\in B$. Krempa called $B$ $\sigma$-rigid if there exists a rigid endomorphism $\sigma$ of $B$. Since Ore extensions of injective type are particular examples of skew PBW  extensions, the first author introduced the following definition with the purpose of studying the notion of {\em rigidness} for this more general setting. 
\begin{definition}[\cite{Reyes2015}, Definition 3.2] \label{generaldef2015}
Let $B$ be a ring and $\Sigma$ a family of endomorphisms of $B$. $\Sigma$ is called a {\em rigid endomorphisms family} if $r\sigma^{\alpha}(r)=0$ implies $r=0$, for every $r\in B$ and $\alpha\in \mathbb{N}^n$. A ring $B$ is called to be $\Sigma$-{\em rigid} if there exists a rigid endomorphisms family $\Sigma$ of $B$.
\end{definition}
Note that if $\Sigma$ is a rigid endomorphisms family, then every element $\sigma_i\in \Sigma$ is a monomorphism. In fact, $\Sigma$-rigid rings are reduced rings: if $B$ is a $\Sigma$-rigid ring and $r^2=0$ for $r\in B$, then $0=r\sigma^{\alpha}(r^2)\sigma^{\alpha}(\sigma^{\alpha}(r))=r\sigma^{\alpha}(r)\sigma^{\alpha}(r)\sigma^{\alpha}(\sigma^{\alpha}(r))=r\sigma^{\alpha}(r)\sigma^{\alpha}(r\sigma^{\alpha}(r))$, i.e., $r\sigma^{\alpha}(r)=0$ and so $r=0$, that is, $B$ is reduced (note that there exists an endomorphism of a reduced ring which is not a rigid endomorphism, see \cite{HongKimKwak2000}, Example 9). With this in mind, we consider the family of injective endomorphisms $\Sigma$ and the family $\Delta$ of $\Sigma$-derivations in a skew PBW extension $A$ over a ring $R$ (see Proposition \ref{sigmadefinition}). Remarkable examples of $\Sigma$-rigid rings can be found in \cite{ReyesSuarezUMA2018}, Examples 3.3, \cite{Reyes2018}, Examples 2.9 or \cite{ReyesSuarezYesica2018}, Example 2. \\

Now, following the ideas presented by Ouyang \cite{Ouyang2008} for Ore extensions, we present the following definition which extends $\Sigma$-rigid rings.
\begin{definition}\label{weaksigmarigidring}
Let $\Sigma=\{\sigma_1,\dotsc,\sigma_n\}$ and $\Delta=\{\delta_1,\dotsc,\delta_n\}$ be a family of endomorphisms and $\Sigma$-derivations of $R$, respectively. $R$ is called a {\em weak} $\Sigma$-{\em rigid ring}, if $a\sigma^{\theta}(a)\in {\rm nil}(R)\Leftrightarrow a\in {\rm nil}(B)$, for each element $a\in R$ and every $\theta\in \mathbb{N}^{n}$.
\end{definition}
\begin{remark}
It is clear that $\Sigma$-rigid rings are weak $\Sigma$-rigid. However, the converse is false as we can appreciated in the following example taken from \cite{Ouyang2008}, Example 2.1. Let $\sigma$ be an endomorphism of a ring $R$ which is an $\sigma$-rigid ring. Consider the ring
\[
R_3:=\biggl\{\begin{pmatrix}
a & b & c \\ 0 & a & d \\ 0 & 0 & a
\end{pmatrix} \mid a, b, c\in R
   \biggr\}.
\]
If we extend the endomorphism $\sigma$ of $R$ to the endomorphism $\overline{\sigma}:R_3\to R_3$ defined by $\overline{\sigma}(a_{ij}) = (\sigma(a_{ij}))$, then $R_3$ is a weak $\overline{\sigma}$-rigid ring but $R_3$ is not $\overline{\sigma}$-rigid. Therefore, weak $\Sigma$-rigid rings are a generalization of $\Sigma$-rigid rings to the case where the ring of coefficients is not assumed to be reduced (note that the ring $R_3$ is not reduced).
\end{remark}

The next theorem gives an equivalence between the notions of $\Sigma$-rigid rings and weak $\Sigma$-rigid rings. This result extends \cite{Ouyang2008}, Proposition 2.2.
\begin{theorem}\label{2008Proposition2.2}
Let $\Sigma=\{\sigma_1,\dotsc,\sigma_n\}$ and $\Delta=\{\delta_1,\dotsc,\delta_n\}$ be a family of endomorphisms and $\Sigma$-derivations of $R$, respectively. $R$ is $\Sigma$-rigid if and only if $R$ is weak $\Sigma$-rigid and reduced.
\end{theorem}
\begin{proof}
Suppose that $R$ is $\Sigma$-rigid. As we saw above, $R$ is reduced. Let us see that $R$ is weak $\Sigma$-rigid. If $a\in {\rm nil}(R)$, then $a=0$, since $R$ is reduced, whence $a\sigma^{\theta}(a)=0\in {\rm nil}(R)$, for all $\theta\in \mathbb{N}^{n}$ and $1\le i\le n$. Now, if $a\sigma^{\theta}(a)\in {\rm nil}(R)$, for $a\in R$ and every $\theta\in \mathbb{N}^{n}$, then $a\sigma^{\theta}(a)=0$, for all $\theta\in\mathbb{N}^{n}$, since $R$ is reduced, and hence $a=0$ because $R$ is $\Sigma$-rigid. Then $R$ is weak $\Sigma$-rigid and reduced.

Conversely, suppose that $R$ is weak $\Sigma$-rigid and reduced, and let $a\sigma^{\theta}(a)=0$, for $a\in R$ and $\theta\in \mathbb{N}^{n}$. Then $a\in {\rm nil}(R)$, since $R$ is weak $\Sigma$-rigid, and so $a=0$ because $R$ is reduced. Therefore $R$ is $\Sigma$-rigid.
\end{proof}
The next proposition extends \cite{Ouyang2008}, Proposition 2.3 (compare also with \cite{Reyes2015}, Lemma 3.3).
\begin{proposition}\label{2008Proposition2.3}
If $R$ is a NI ring which is weak $\Sigma$-rigid, then we have the following assertions:
\begin{enumerate}
\item [\rm (1)] If $ab\in {\rm nil}(R)$, then $a\sigma^{\alpha}(b), \sigma^{\beta}(a)b\in {\rm nil}(R)$, for every elements $\alpha, \beta\in \mathbb{N}^{n}$.
\item [\rm (2)] If $\sigma^{\alpha}(a)b\in {\rm nil}(B)$, for some element $\alpha\in \mathbb{N}^{n}$, then $ab, ba\in {\rm nil}(R)$.
\item [\rm (3)] If $a\sigma^{\alpha}(b)\in {\rm nil}(B)$, for some element $\alpha\in \mathbb{N}^{n}$, then $ab, ba\in {\rm nil}(R)$.
\end{enumerate}
\end{proposition}
\begin{proof}
(1) Let $ab\in {\rm nil}(R)$. Using that $\sigma_i(ab)=\sigma_i(a)\sigma_i(b)\in {\rm nil}(R)$, for every $1\le i\le n$, where ${\rm nil}(R)$ is an ideal of $R$, we obtain $b\sigma_i(a)\sigma_i(b)\sigma_i^{2}(a)=b\sigma_i(a)\sigma_i(b\sigma_i(a))\in {\rm nil}(R)$ which shows that $b\sigma_i(a)\in {\rm nil}(R)$ whence $\sigma_i(a)b\in {\rm nil}(R)$, for all $i$. If we consider repeatedly this argument, then we obtain that $\sigma^{\beta}(a)b\in {\rm nil}(R)$, for all $\beta\in \mathbb{N}^{n}$. In a similar way, if $ab\in {\rm nil}(R)$, then $ba\in {\rm nil}(R)$, and so $\sigma^{\alpha}(b)a\in {\rm nil}(R)$ which implies that $a\sigma^{\alpha}(b)\in {\rm nil}(R)$, for every element $\alpha \in \mathbb{N}^{n}$.

(2) Suppose that $\sigma^{\alpha}(a)b\in {\rm nil}(B)$, for some element $\alpha\in \mathbb{N}^{n}$. We have $\sigma^{\alpha}(a)\sigma^{\alpha}(b) = \sigma^{\alpha}(ab) = \sigma_1^{\alpha_1}(\sigma_2^{\alpha_2}(\dotsb (\sigma^{\alpha_n}(ab)))) = \sigma_1(\sigma_1^{\alpha_1-1}(\sigma_2^{\alpha_2}(\dotsb (\sigma^{\alpha_n}(ab))))) \in {\rm nil}(R)$, by part (1). Since ${\rm nil}(R)$ is an ideal of $R$, $(\sigma_1^{\alpha_1-1}(\sigma_2^{\alpha_2}(\dotsb (\sigma^{\alpha_n}(ab)))))\cdot \sigma_1(\sigma_1^{\alpha_1-1}(\sigma_2^{\alpha_2}(\dotsb (\sigma^{\alpha_n}(ab)))))\in {\rm nil}(R)$,  whence we obtain $\sigma_1^{\alpha_1-1}(\sigma_2^{\alpha_2}(\dotsb (\sigma^{\alpha_n}(ab))))$ by definition of weak $\Sigma$-rigid ring. Continuing in this way we can prove that $\sigma_2^{\alpha_2}(\dotsb (\sigma^{\alpha_n}(ab)))\in {\rm nil}(R)$. Again, continuing this procedure we can see that $ab\in {\rm nil}(R)$.

(3) The proof uses a similar argument to the considered in part (2).
\end{proof}
The next proposition generalizes \cite{Ouyang2008}, Proposition 2.4 (compare also with \cite{Reyes2015}, Proposition 3.5).
\begin{proposition}\label{2008Proposition2.4}
If $R$ is a NI and weak $\Sigma$-rigid ring, then $\sigma_i(e) = e\ (1\le i\le n)$, for every central idempotent element $e\in R$.
\end{proposition}
\begin{proof}
Consider $e$ a central idempotent of $R$. It is clear that $e(1-e)=0$. By Proposition \ref{2008Proposition2.3} (1) we obtain $\sigma_i(e)(1-e)\in {\rm nil}(R)$, for $1\le i\le n$. This means that there exists some positive integer $k$ such that $0=(\sigma_i(e)(1-e))^{k} = \sigma_i(e)(1-e)$ (for a fixed $i$). In this way $\sigma_i(e) = \sigma_i(e)e$, for all $i$. In a similar way, $(1-e)e=0\Rightarrow \sigma_i(1-e)e=0$, and so $e=\sigma_i(e)e$, whence $\sigma_i(e)=e$, for all $i$.
\end{proof}
With the aim of establishing the following proposition, an ideal $I$ of $R$ will be called {\em weak} $\Sigma$-rigid, if $a\sigma^{\theta}(a)\in {\rm nil}(R) \Leftrightarrow a\in {\rm nil}(R)$, for every element $a\in I$ and each $\theta\in \mathbb{N}^{n}$. Our Proposition \ref{2008Proposition2.5} extends \cite{Ouyang2008}, Proposition 2.5.
\begin{proposition}\label{2008Proposition2.5}
If $R$ is an abelian ring with $\sigma_i(e)=0$ $(1\le i\le n)$, for every idempotent element $e$ of $R$, then the following assertions are equivalent:
\begin{enumerate}
\item [\rm (1)] $R$ is weak $\Sigma$-rigid.
\item [\rm (2)] $eR$ and $(1-e)R$ are weak $\Sigma$-rigid ideals.
\end{enumerate}
\end{proposition}
\begin{proof}
$(1)\Rightarrow (2)$ It is clear since $eR$ and $(1-e)R$ are subrings of $R$.

$(2)\Rightarrow (1)$ Let $a$ be a nilpotent element of $R$. Then $ea, (1-e)a\in {\rm nil}(R)$. Having in mind that $eR$ and $(1-e)R$ are weak $\Sigma$-rigid, there exist positive integers $k, l$ with $(ea\sigma_i(ea))^{k} = e(a\sigma_i(a))^{k} = 0$ and $((1-e)a\sigma_i((1-e)a))^{l} = (1-e)(a\sigma_i(a))^{l}$, for a fixed $i$. If we take $m:={\rm max}\{k, l\}$, then $e(a\sigma_i(a))^{m} = (1-e)(a\sigma_i(a))^{m} = 0$. Therefore $(a\sigma_i(a)))^{m} = 0$, that is, $a\sigma_i(a)\in {\rm nil}(R)$, for all $i$.

Conversely, suppose that $a\sigma^{\theta}(a)\in {\rm nil}(R)$, for $\theta\in \mathbb{N}^{n}$. Then $ea\sigma^{\theta}(ea)\in {\rm nil}(R)$ and $(1-e)a\sigma^{\theta}((1-e)a) \in {\rm nil}(R)$. So, $ea\in {\rm nil}(R)$ and $(1-e)a\in {\rm nil}(R)$, since $eR$ and $(1-e)R$ are weak $\Sigma$-rigid ideals. Hence $a\in {\rm nil}(R)$, that is, $R$ is weak $\Sigma$-rigid.
\end{proof}
\section{Weak $\Sigma$-skew Armendariz rings}\label{weakSigmaskewArmendariz}
In the literature we find the following notions about Armendariz rings in commutative and noncommutative case concerning Ore extensions.
\begin{definition}\label{cardona}
\begin{enumerate}
\item [\rm (i)] (\cite{LiuZhao2006}, Definition 2.1) A ring $B$ is called {\em weak Armendariz}, if whenever polynomials $f=\sum_{i=0}^{s} a_ix^{i}$ and $g=\sum_{j=0}^{t} b_jx^{j} \in B[x]$ satisfy $pq=0$, then $a_ib_j\in {\rm nil}(B)$, for each $i, j$.
\item [\rm (ii)] (\cite{HongKimKwak2003}, p. 104) $B$ is called $\alpha$-{\em skew Armendariz}, if whenever $f=\sum_{i=0}^{s} a_ix^{i}$ and $g=\sum_{j=0}^{t}b_jx^{j}\in B[x;\alpha]$ with $fg=0$, then  $a_i\sigma^{i}(b_j)=0$, for every $i, j$.
\item [\rm (iii)] (\cite{Ouyang2008}, p. 110) $B$ is called {\em weak} $\alpha$-skew Armendariz, if whenever $f=\sum_{i=0}^{s} a_ix^{i}$ and $g=\sum_{j=0}^{t} b_jx^{j}\in B[x;\alpha]$ satisfy $fg=0$, then $a_i\alpha^{i}(b_j)\in {\rm nil}(B)$.\end{enumerate}
\end{definition}
In the context of skew PBW extensions, the authors have defined the following Armendariz notions:
\begin{definition}
Let $A$ be a skew PBW extension over a ring $R$. Then:
\begin{enumerate}
\item [\rm (i)] (\cite{NinoReyes2017}, Definition 3.4) $R$ is called a $(\Sigma, \Delta)$-{\em skew Armendariz ring}, if whenever $f = \sum_{i=0}^{t} a_iX_i$, $g = \sum_{j=0}^{s} b_jY_j\in A$ with $fg=0$, then $a_iX_ib_jY_j=0$, for every value of $i$ and $j$.
\item [\rm (ii)](\cite{ReyesSuarez2016UIS}, Definition 3.1) $R$ is called a $\Sigma$-{\em skew Armendariz ring}, if for elements $f=\sum_{i=0}^{m} a_iX_i$ and $g=\sum_{j=0}^{t} b_jY_j$ in $A$, the equality $fg=0$ implies $a_i\sigma^{\alpha_i}(b_j) = 0$, for all $0\le i\le m$ and $0\le j\le t$, where $\alpha_i = {\rm exp}(X_i)$.
\item [\rm (iii)] (\cite{ReyesSuarezClifford2017}, Definition 4.1) $R$ is a {\em skew-Armendariz} ring, if for polynomials $f=a_0+a_1X_1+\dotsb + a_mX_m$ and $g=b_0+b_1Y_1 + \dotsb + b_tY_t$ in $A$, $fg=0$ implies $a_0b_k=0$, for each $0\le k\le t$. 
\item [\rm (iv)] (\cite{Reyes2018}, Definition 3.1) $R$ is called a {\em skew}-$\Pi$ {\em Armendariz ring}, if for elements $f=\sum_{i=0}^{m} a_iX_i,\ g=\sum_{j=0}^{t} b_jY_j$ of $A$, $fg\in {\rm nil}(A)$ implies that $a_ib_j\in {\rm nil}(R)$, for every $0\le i\le m$ and $0\le j\le t$.
\end{enumerate}
\end{definition}
Several relations about these four notions of Armendariz rings for coefficient rings of skew PBW extensions can be found in \cite{ReyesSuarez2016UIS}, Section 3, \cite{ReyesSuarezClifford2017}, Sections 3 and 4, and \cite{Reyes2018}, Section 3. Now, with the aim of extending Definition \ref{cardona} (iii) from Ore extensions of endomorphism type to skew PBW extensions of endomorphism type (which are more general, see Examples \ref{mentioned}), and also $\Sigma$-skew Armendariz rings defined by the first author in \cite{ReyesSuarez2016UIS}, Definition 3.1, (at least in the endomorphism case), we present the following definition.
\begin{definition}\label{weaksigmaskewArmendarizPBW}
Let $A$ be a skew PBW extension of endomorphism type over  a ring $R$. $R$ is called a {\em weak} $\Sigma$-{\em skew Armendariz ring}, if for elements $f=\sum_{i=0}^{m} a_iX_i$ and $g=\sum_{j=0}^{t} b_jY_j$ in $A$, the equality $fg=0$ implies $a_i\sigma^{\alpha_i}(b_j)\in {\rm nil}(R)$, for all $0\le i\le m$ and $0\le j\le t$, where $\alpha_i = {\rm exp}(X_i)$.
\end{definition}
The following theorem extends \cite{Ouyang2008}, Theorem 3.3. We need to assume that the elements $c_{i,j}$ in Definition \ref{2008Proposition2.3} (iv) are both central and invertible in $R$. We denote ${\rm nil}(R)A:=\{f\in A\mid f= a_0 + a_1X_1 + \dotsb + a_mX_m,\ a_i\in {\rm nil}(R)\}$.
\begin{theorem}\label{2008Theorem3.3}
If $R$ is a NI and weak $\Sigma$-rigid ring, then $R$ is a weak $\Sigma$-skew Armendariz ring.
\end{theorem}
\begin{proof}
Suppose that $fg=0$, where $f=a_0+a_1X_1+\dotsb + a_mX_m$ and $g=b_0+b_1Y_1+\dotsb + b_tY_t$, with the monomial order $X_1\prec X_2\prec \dotsb \prec X_m$ and $Y_1\prec Y_2\prec \dotsb \prec Y_t$, respectively (Definition \ref{definitioncoefficients} (ii)). Since $fg =\sum_{k=0}^{m+t} \biggl(\sum_{i+j=k} a_iX_ib_jY_j\biggr)$, then ${\rm lc}(fg)= a_m\sigma^{\alpha_m}(b_t)c_{\alpha_m, \beta_t}=0$. By assumption, the elements $c_{i,j}$ (Definition \ref{2008Proposition2.3} (iv)) are invertible in $R$, so $c_{\alpha_m,\beta_t}$ are also invertible, and hence ${\rm lc}(fg)= a_m\sigma^{\alpha_m}(b_t)=0$ which means that the element $a_m\sigma^{\alpha_m}(b_t)\in {\rm nil}(R)$. The idea is to prove that $a_p\sigma^{\alpha_p}(b_q)\in {\rm nil}(R)$,  for $p+q\ge 0$.  We proceed  by induction. Suppose that $a_p\sigma^{\alpha_p}(b_q)\in {\rm nil}(R)$, for $p+q=m+t, m+t-1, m+t-2, \dotsc, k+1$ for some $k>0$. From Remark \ref{juradpr} and Proposition \ref{2008Proposition2.3} we obtain that  $a_pX_pb_qY_q$ is an element of ${\rm nil}(R)A$, for these values of $p+q$. In this way we only consider the sum of the products $a_uX_ub_vY_v$, where $u+v=k, k-1,k-2,\dotsc, 0$. Fix $u$ and $v$. Consider the sum of all terms  of $fg$  having exponent $\alpha_u+\beta_v$. By Proposition \ref{lindass}, Remark \ref{juradpr} and the assumption $fg=0$, we know that the sum of all coefficients of all these terms  can be written as
\begin{equation}\label{Federer}
a_u\sigma^{\alpha_u}(b_v)c_{\alpha_u, \beta_v} + \sum_{\alpha_{u'} + \beta_{v'} = \alpha_u + \beta_v} a_{u'}\sigma^{\alpha_{u'}} ({\rm \sigma's\ and\ \delta's\ evaluated\ in}\ b_{v'})c_{\alpha_{u'}, \beta_{v'}} = 0.
\end{equation}
By assumption we know that $a_p\sigma^{\alpha_p}(b_q)\in {\rm nil}(R)$, for $p+q=m+t, m+t-1, \dotsc, k+1$.  So, Proposition \ref{2008Proposition2.3} (1) guarantees that the product 
\[a_p({\rm \sigma's\ and\ \delta's\ evaluated\ in}\ b_{q})\ \ \ \ \ \ \ ({\rm any\  order\ of}\ \sigma's\ {\rm and}\ \delta's)
\]
is an element of ${\rm nil}(R)$. Proposition  \ref{2008Proposition2.3} guarantees that $({\rm \sigma's\ and\ \delta's\ evaluated\ in}\ b_{q})a_p$ is also an element of ${\rm nil}(R)$. In this way, multiplying (\ref{Federer}) by $a_k$, and using the fact that the elements $c_{i,j}$ in Definition \ref{gpbwextension} (iv) are in the center of $R$, 
\begin{equation}\label{doooooo}
a_u\sigma^{\alpha_u}(b_v)a_kc_{\alpha_u, \beta_v} + \sum_{\alpha_{u'} + \beta_{v'} = \alpha_u + \beta_v} a_{u'}\sigma^{\alpha_{u'}} ({\rm \sigma's\ and\ \delta's\ evaluated\ in}\ b_{v'})a_kc_{\alpha_{u'}, \beta_{v'}} = 0,
\end{equation}
whence, $a_u\sigma^{\alpha_u}(b_0)a_k=0$. Since $u+v=k$ and $v=0$, then $u=k$, so $a_k\sigma^{\alpha_k}(b_0)a_k=0$ whence $a_k\sigma^{\alpha_k}(b_0)\in {\rm nil}(R)$ by Proposition \ref{2008Proposition2.3}. Therefore, we now have to study the expression (\ref{Federer}) for $0\le u \le k-1$ and $u+v=k$. If we multiply (\ref{doooooo}) by $a_{k-1}$ we obtain 
{\small{\[
a_u\sigma^{\alpha_u}(b_v)a_{k-1}c_{\alpha_u, \beta_v} + \sum_{\alpha_{u'} + \beta_{v'} = \alpha_u + \beta_v} a_{u'}\sigma^{\alpha_{u'}} ({\rm \sigma's\ and\ \delta's\ evaluated\ in}\ b_{v'})a_{k-1}c_{\alpha_{u'}, \beta_{v'}} = 0.
\]}}
Using a similar reasoning as above, we can see that $a_u\sigma^{\alpha_u}(b_1)a_{k-1}c_{\alpha_u, \beta_1}=0$, and using the assumptions on the elements $c_{\alpha_u,\beta_1}$. Now, since $a_u\sigma^{\alpha_u}(b_1)a_{k-1}=0$, and $u=k-1$, Proposition \ref{2008Proposition2.3} imply that $a_{k-1}\sigma^{\alpha_{k-1}}(b_1)=0$. Continuing in this way, we prove that $a_i\sigma^{\alpha_i}(b_j)\in {\rm nil}(R)$, for $i+j=k$. Therefore $a_i\sigma^{\alpha_i}(b_j)\in {\rm nil}(R)$, for $0\le i\le m$ and $0\le j\le t$.
\end{proof}
\begin{remark}\label{2008Example3.4}
The importance of the condition NI on $R$ in Theorem \ref{2008Theorem3.3} can be appreciated in the following example taken from \cite{Ouyang2008}, Example 3.4, which presents a noncommutative ring which is weak $\Sigma$-rigid but not weak $\Sigma$-skew Armendariz. Let $R$ be a ring and $M_2(R)$ be the $2\times 2$ matrix ring over $R$. Let
\[
S = \biggl \{\begin{pmatrix} A & B\\ 0 & C\end{pmatrix}\mid A, B, C \in M_2(R)  \biggr\}.
\]
It is clear that $S$ is a ring with usual matrix operations. If we consider the endomorphism $\sigma:S\to S$ defined by 
\[
\sigma \biggl ( \begin{pmatrix} A & B\\ 0 & C\end{pmatrix}\biggr) = \biggl (\begin{pmatrix} A & -B\\ 0 & C\end{pmatrix} \biggr),\ \ \ \ \begin{pmatrix} A & B\\ 0 & C\end{pmatrix} \in S,
\] 
then $S$ is weak $\sigma$-rigid but not weak $\sigma$-skew Armendariz.
\end{remark}
\section{Future work}\label{futurework}
Having in mind that $\Sigma$-rigid rings have been studied in several papers concerning ring theoretical properties such as Armendariz, Baer, quasi-Baer, p.p. and p.q.-Baer rings, zip, McCoy, invariant ideals, ascending chain condition on principal left (resp. right) ideals, and others (c.f. \cite{Reyes2015}, \cite{ReyesSuarez2016Boletin}, \cite{ReyesSuarez2016UIS}, \cite{NinoReyes2017}, \cite{ReyesSuarezClifford2017}, \cite{ReyesSuarezUMA2018},  \cite{Reyes2018} and \cite{ReyesSuarezYesica2018}), there is a considerable number of results about $\Sigma$-rigid rings which can be extended to the more general setting of weak $\Sigma$-rigid rings. This will be our line of thinking in future papers.

\vspace{0.5cm}

\noindent {\bf \Large{Acknowledgements}}

\vspace{0.5cm}

The first author was supported by the research fund of Facultad de Ciencias, Universidad Nacional de Colombia, Bogot\'a, Colombia, HERMES CODE 41535.


\end{document}